\documentclass{amsart}

\usepackage{amsaddr}
\usepackage{amsmath, amssymb, amsthm, bm}
\usepackage{color}
\usepackage{hyperref}
\usepackage{url}
\usepackage[utf8]{inputenc}

\usepackage{enumerate}
\usepackage{graphicx}
\usepackage{appendix}

\newtheorem{thm}{Theorem}[section]
\newtheorem{prop}[thm]{Proposition}
\newtheorem{lem}[thm]{Lemma}
\newtheorem{cor}[thm]{Corollary}

\newtheorem{lemma}[thm]{Lemma}

\newtheorem{example}[thm]{Example}
\newtheorem{remark}[thm]{Remark}
\newtheorem{conj}[thm]{Conjecture}

\newenvironment{mydef}[1][Definition]{\begin{trivlist}
\item[\hskip \labelsep {\bfseries #1}]}{\end{trivlist}}

\newcommand{\m}[1]{\mathcal{#1}}

\newcommand{\mb}{\mathbb}
\newcommand{\F}{\mathbb{F}}

\newcommand{\Z}{\mathbb{Z}}
\newcommand{\N}{\mathbb{N}}
\newcommand{\ra}{\rightarrow}

\newcommand{\mm}{\mathfrak{m}}
\newcommand{\uu}{\# \mathcal{U}}

\begin{document}
\title{On zeros of a polynomial in a finite grid}
\author{Anurag Bishnoi}
\address{Department of Mathematics, Ghent University, Ghent, Belgium}
\email{anurag.2357@gmail.com}

\author{Pete L. Clark}
\address{Department of Mathematics, University of Georgia, Athens, GA, USA}
\email{plclark@gmail.com}

\author{Aditya Potukuchi}
\address{Department of Computer Science, Rutgers University, NJ, USA}
\email{apotu.57@gmail.com}

\author{John R. Schmitt}
\address{Department of Mathematics, Middlebury College, Middlebury, VT, USA}
\email{jschmitt@middlebury.edu}
\maketitle

\begin{abstract}
\noindent
A 1993 result of Alon and F\"uredi gives a sharp upper bound on  the number of zeros of a multivariate polynomial over an integral domain in a finite grid, in terms of the degree of the polynomial.  This result was recently generalized to polynomials over an arbitrary commutative ring, assuming a certain ``Condition (D)'' on the grid which holds vacuously when the ring is a
domain.  In the first half of this paper we give a further Generalized Alon--F\"uredi Theorem which provides a sharp upper bound when the degrees of the polynomial in each variable are also taken into account.  This yields in particular a new proof of Alon--F\"uredi.  We then discuss the relationship between Alon--F\"uredi and results of DeMillo--Lipton, Schwartz and Zippel.    A direct coding theoretic interpretation of Alon--F\"uredi Theorem and its generalization in terms of Reed--Muller type affine variety codes is shown which gives us the minimum Hamming distance of these codes. Then we apply the Alon--F\"uredi Theorem to quickly recover – and sometimes strengthen – old and new results in finite geometry, including the Jamison/Brouwer--Schrijver bound on affine blocking sets.  We end with a discussion of multiplicity enhancements.
\end{abstract}


\tableofcontents

\section{Introduction}
\subsection{Notation} We denote the positive integers by $\Z^+$ and the non-negative integers by $\N$.  For $n \in \Z^+$, we put $[n] = \{1, 2, \dots, n\}$.
\\ \\
For us, rings are commutative with multiplicative identity. Throughout this paper $R$ denotes a ring and $F$ denotes a field, each arbitrary unless otherwise specified.
\\ \\
Following \cite{Schauz08} and \cite{Clark14}, a nonempty subset $S \subset R$ is said to satisfy \textbf{Condition (D)} if for all $x \neq y \in S$, the element $x - y \in R$ is not a zero divisor.  A \textbf{finite grid} is a subset $A = \prod_{i=1}^n A_i$ of $R^n$ (for some $n \in \Z^+$) with each $A_i$ a finite, nonempty subset of $R$.  We say that $A$ satisfies Condition (D) if each $A_i$ does.
\\ \\
For $A \subset R^n$ and $f \in R[\underline{t}] = R[t_1,\ldots,t_n]$, we put
\[ Z_A(f) = \{x \in A \mid f(x) = 0 \} \text{ and } \ \mathcal{U}_A(f) = \{x \in A \mid f(x) \neq 0 \}. \]

\subsection{The Alon--F\"uredi Theorem} 
In \cite{Alon--Furedi93} Alon and F\"{u}redi solved a problem posed by B\'{a}r\'{a}ny (based on a result of Komjath) of finding the minimum number of hyperplanes required to cover all points of the hypercube $\{0, 1\}^n \subseteq F^n$ \textit{except one}.  One such covering is given by $n$ hyperplanes defined by the zeros of the polynomials $t_1 -1$, $t_2 -1$, $\ldots$, $t_n -1$.
Alon and F\"{u}redi proved that $n$ is in fact the minimum number.  They then generalised this result to all finite grids $A = \prod_{i=1}^n A_i \subset F^n$, showing that the minimum number of hyperplanes required to cover all points of $A$ except one is $\sum_{i = 1}^n (\#A_i - 1)$.
\\ \\
There is also a quantitative refinement: as we vary over families of $d$ hyperplanes which do not cover all
points of $A$, what is the minimum number of points which are missed?  To answer this, Alon and F\"uredi proved the following result.

\begin{thm}[{Alon--F\"{u}redi Theorem \cite[Thm. 5]{Alon--Furedi93}}]
\label{thm:AF}
Let $F$ be a field, let $A = \prod_{i=1}^n A_i \subset F^n$ be a finite grid, and let $f \in F[\underline{t}] = F[t_1,\ldots,t_n]$ be a polynomial which does not vanish on all points of $A$.  Then $f(x) \neq 0$ for
at least $\min \prod y_i$ elements $x \in A$, where the minimum is taken over all positive integers $y_i \leq  \# A_i$ with $\sum_{i=1}^n y_i = \sum_{i=1}^n \# A_i - \deg f$.  More concisely (see $\S$\ref{sec:balls})
\[ \#\mathcal{U}_A(f) \geq \mm(\# A_1,\ldots,\# A_n;\sum_{i=1}^n \# A_i - \deg f). \]
\end{thm}
\noindent
The minimum referred to in Theorem \ref{thm:AF} is known in all cases -- see Lemma \ref{FSLEMMA}a) -- leading to an explicit form of the bound.
\\ \\
Several proofs of Theorem \ref{thm:AF} have been given.  The original argument in \cite{Alon--Furedi93} involves the construction of auxiliary polynomial functions of low degree via linear algebra.  A second proof was given by Ball and Serra as an application of their Punctured Combinatorial Nullstellensatz \cite{Ball-Serra09}, \cite{Ball-Serra11}.  
Recently, L\'{o}pez, Renter\'{a}-M\'{a}rquez and Villarreal gave a proof of Alon--F\"uredi \cite{LRMV14}, in its coding theoretic formulation (see $\S$\ref{sec:coding_theory}).
Geil had noticed that the minimum distance of generalized Reed--Muller codes can be determined easily using the Gr\"obner basis theory \cite[Theorem 2]{Geil08}.
This technique was then used by Carvalho to give another proof of Theorem \ref{thm:AF} when $F$ is a finite field \cite[Prop. 2.3]{Carvalho13}, which is in fact a special case of an earlier result by Geil and Thomsen \cite[Prop. 5]{Geil-Thomsen13} (take all weights equal to $1$). 

In \cite{Clark15}, Clark generalized the Alon--F\"uredi Theorem by replacing the field $F$ by an arbitrary ring $R$, under the assumption that the finite grid $A$ satisfies Condition (D).  This is a modest generalization in that Condition (D) is exactly what is needed for polynomial functions on $A$ to behave as they do in the case of a field, and the proof adapts that of Ball-Serra.
\\ \\
Clark, Forrow and Schmitt \cite{CFS16} used Alon--F\"uredi to obtain a restricted variable generalization of a theorem of Warning \cite{Warning35} giving a lower bound on the number of zeros of a system of polynomials over a finite field.  (Alon--F\"uredi gives a lower bound on \emph{non-zeros}, but over a finite field $\F_q$,
we have Chevalley's trick: $f(x) = 0 \iff 1-f(x)^{q-1} \neq 0$.)  This work also gave several combinatorial
applications of this lower bound on restricted variable zero sets.
\\ \\
One of the main goals of this paper is to revisit the Alon--F\"uredi Theorem and give direct combinatorial applications (i.e., not of Chevalley--Warning type).
We begin by giving the following generalization of the Alon--F\"uredi Theorem:

\begin{thm}[Generalized Alon--F\"{u}redi Theorem]
\label{thm:main}
Let $R$ be a ring and let $A_1, \dots, A_n$ be non-empty finite subsets of $R$ that satisfy Condition (D).  For $i \in [n]$, let $b_i$ be an integer such that $1 \leq b_i \leq \# A_i$.
Let $f \in R[t_1, \dots, t_n]$ be a non-zero polynomial such that $\deg_{t_i} f \leq \# A_i - b_i$ for all $i \in [n]$.
Let $\m U_A(f) = \{x \in A \mid f(x) \neq 0\}$ where $A = A_1 \times \dots \times A_n \subset R^n$.
Then we have (see $\S$\ref{sec:balls})
\[\# \m U_A(f) \geq \mm(\# A_1, \dots, \# A_n; b_1, \dots b_n; \sum_{i = 1}^n \# A_i - \deg f).\]
Moreover, this bound is sharp in all cases.
\end{thm}
\noindent
As we shall explain in $\S$\ref{sec:GAF}, one recovers Theorem \ref{thm:AF} from Theorem \ref{thm:main} by taking $b_1 = \ldots = b_n = 1$.  Our argument specializes to give a new proof of Alon--F\"uredi.
\\ \\
In $\S$\ref{sec:SZ_main} we relate the Generalized Alon--F\"uredi Theorem to work of DeMillo--Lipton, Schwartz and Zippel.  We find in particular that Alon--F\"uredi implies the result which has become known as the ``Schwartz--Zippel Lemma''.  In fact, the original result of Zippel (and earlier, DeMillo--Lipton) is a bit different and not implied by Alon--F\"uredi (see Example \ref{EXAMPLE2}).  However, it is implied by Generalized Alon--F\"uredi, and this was one of our motivations for strengthening Alon--F\"uredi as we have.
\\ \\
The Alon--F\"uredi Theorem has a natural  coding theoretic interpretation (see $\S$\ref{sec:coding_theory}) as it computes the minimum Hamming distance of the affine grid code ${\rm AGC}_d(A)$, an $F$-linear code of
length $\# A$.  In this way Alon--F\"uredi turns out to be the restricted variable generalization of a much older result in the case $A_i = F = \F_q$, the
Kasami--Lin--Peterson Theorem, which computes the minimum Hamming distance of generalized Reed--Muller codes.
We will show that the Generalized Alon--F\"uredi Theorem is equivalent to the computation of the minimum Hamming distance of a more general class of $R$-linear codes.   These generalized affine grid codes have larger distance (though also smaller dimension) than the standard ones, so they may turn out to be useful.
\\ \\
In $\S$\ref{sec:finite_geometry}, we pursue applications to finite geometry.  We begin by revisiting and slightly sharpening the original result of Alon--F\"uredi on hyperplane coverings.  This naturally leads us to partial covers and blocking sets in affine and projective geometries over $\F_q$.  Applying Alon--F\"uredi and projective duality we get a new upper bound, Theorem \ref{thm:affine}, on the number of hyperplanes which do not meet a $k$-element subset of $\mathrm{AG}(n,q)$.  From this result the classical theorems of Jamison--Brouwer--Schrijver on affine blocking sets and Blokhuis--Brouwer on essential points of projective blocking sets follow as corollaries.  We are also able to strengthen a recent result of Dodunekov, Storme and Van de Voorde.
\\ \\
Finally, in $\S$\ref{sec:mult} we discuss multiplicity enhancements in the sense of \cite{DKSS13}.  The material here is most closely related to that of $\S$\ref{sec:SZ_main}, but we have placed it at the end because it has a somewhat more technical character than the rest of the paper.  


\subsection{Acknowledgements}
We wish to thank Olav Geil and Fedor Petrov for helpful conversations and useful pointers to the literature.  

\section{Preliminaries}

\subsection{Balls in Prefilled Bins}
\label{sec:balls}
Let $a_1, \dots, a_n \in \Z^+$.  Consider $n$ bins $A_1, \dots, A_n$ such that $A_i$ can hold up to $a_i$ balls.  For $N \in \Z^+$ with $n \leq N \leq \sum_{i=1}^n a_i$, we define a distribution of $N$ balls in these $n$ bins to be an $n$-tuple $y = (y_1, \dots, y_n) \in (\Z^+)^n$ with $y_i \leq a_i$ for all $i \in [n]$ and $\sum_{i=1}^n y_i = N$.
For a distribution $y$ of $N$ balls in $n$ bins, we put $P(y) = \prod_{i=1}^n y_i$.
For $n \leq N \leq \sum_{i=1}^n a_i$ we define $\mm(a_1, \dots, a_n; N)$ to be the minimum value of $P(y)$ as $y$ ranges over all such distributions of $N$ balls in $n$ bins.  For $N < n$ we define $\mm(a_1, \dots, a_n; N) = 1$.
\\ \\
Without loss of generality we may -- and shall -- assume $a_1 \geq \dots \geq a_n$.  We define the
\textbf{greedy distribution} $y_G = (y_1,\ldots,y_n)$ as follows: first place one ball in each bin; then place the remaining balls into bins from left to right, filling each bin completely before moving on to the next bin, until we run out of balls.

\begin{lemma}
\label{lem:FS}
\label{FSLEMMA}
Let $n \in \mb Z^+$, and let $a_1 \geq \dots \geq a_n$ be positive integers.
Let $N \in \Z$ with $n \leq N \leq a_1 + \dots + a_n$.  
\begin{enumerate}[$(a)$]
\item We have
\[ \mm(a_1,\ldots,a_n;N) = P(y_G) = y_1 \cdots y_n. \]
\item Suppose $a_1 = \dots = a_n = a \geq 2$.  Then
\[ \mm(a,\ldots,a;N) = (r+1)a^{\left\lfloor \frac{N-n}{a-1} \right\rfloor}, \]
where $r \equiv N-n \pmod{a-1}$ and $0 \leq r < a-1$.  
\item For all non-negative integers $k$, we have
\[\mm(2,\ldots,2;2n-k) = 2^{n-k}. \]
\item Let $n, a_1,\ldots,a_n \in \Z^+$ with $a_1 \geq \dots \geq a_n$.
Let $N \in \Z$ be such that $N - n = \sum_{i = 1}^j (a_i - 1) + r$ for some $j \in \{0, \dots, n\}$ and some $r$ satisfying $0 \leq r < a_{j+1}$.
Then $\mathfrak{m}(a_1, \dots, a_n; N) = (r+1)\prod_{i = 1}^j a_i$.
\end{enumerate}
\end{lemma}
\begin{proof}
Parts $(a)$ through $(c)$ are \cite[Lemma 2.2]{CFS16}.  $(d)$ After placing one ball in each bin we are left with $N - n$ balls.  The greedy distribution is achieved by filling the first $j$ bins entirely and then putting $r$ balls in bin $j+1$.
\end{proof}
\noindent
In every distribution $y = (y_1,\dots,y_n)$ we need $y_i \geq 1$ for all $i \in [n]$; i.e., we must place at least one ball in each bin.  So it is reasonable to think of the bins coming \textbf{prefilled} with one ball each, and then our task is to distribute the $N-n$ remaining balls
so as to minimize $P(y)$.  The concept of prefilled bins makes sense more generally: given any
$b_1,\dots,b_n \in \Z$ with $1 \leq b_i \leq a_i$, we may consider the scenario in which the $i$-th bin comes prefilled with $b_i$ balls.  If $\sum_{i=1}^n b_i \leq N \leq \sum_{i=1}^n a_i$, we may restrict to distributions $y = (y_1,\dots,y_n)$ of $N$ balls into bins of sizes $a_1,\dots,a_n$ such that $b_i \leq y_i \leq a_i$ for all $i \in [n]$ and put
\[ \mm(a_1,\dots,a_n;b_1,\dots,b_n;N) = \min P(y), \]
where the minimum ranges over this restricted set of
distributions.
For $N < \sum_{i = 1}^n b_i$ we define $\mm(a_1, \dots, a_n; b_1, \dots, b_n; N) := \prod_{i = 1}^n b_i$.

\begin{lemma}
\label{LEMMA2.2}
We have $\mm(a_1,\dots,a_n;b_1,\dots,b_n;N) = \prod_{i = 1}^n b_i \iff N \leq \sum_{i=1}^n b_i$.
\end{lemma}
\begin{proof}
If $N \leq \sum_{i=1}^n b_i$ then $\mm(a_1,\dots,a_n;b_1,\dots,b_n;N) = \prod_{i=1}^n b_i$ by definition unless $N = \sum_{i=1}^n b_i$, and this case is immediate: we have exactly enough balls to perform the prefilling.  If $N > \sum_{i=1}^n b_i$, then $\mm(a_1,\dots,a_n;b_1,\dots,b_n;N)$ is
the minimum over a set of integers each of which is strictly greater than $\prod_{i=1}^n b_i$.
\end{proof}
\noindent
In this prefilled context, the greedy distribution $y_G$ is defined by starting with the bins prefilled with $b_1,\dots,b_n$ balls and then distribute the remaining balls from left to right, filling each bin completely before moving on to the next bin.  One sees, e.g. by adapting the argument of \cite[Lemma 2.2]{CFS16}, that
\[ \mm(a_1,\dots,a_n;b_1,\dots,b_n;N) = P(y_G) \]
when we also have $b_1 \geq \dots \geq b_n$.
But this may not hold in general, as the following example shows.

\begin{example}
\label{EXAMPLE1}
Let $n = 2$, $a_1 = 4$, $a_2 = 3$, $b_1 = 1$, $b_2 = 2$, $N = 4$.  Then $P(y_G) = 4$ but $\mm(4,3;1,2;4) = 3$ achieved by the distribution $(1, 3)$.
\end{example}
\noindent
In general, we do not know a simple description of $\mm(a_1,\dots,a_n;b_1,\dots,b_n;N)$.  In practice, it can be computed using dynamic programming.

\begin{lem}
\label{lem:main}
Let $a_1, \dots, a_n, b_1, \dots, b_n \in \Z^+$
with $1 \leq b_i \leq a_i$ for all $i \in [n]$.
Let $k \in \mb Z$ such that $b_n \leq k \leq a_n$.
If \[b_1 + \dots + b_{n-1} \leq N - k \leq a_1 + \dots + a_{n-1}\] for some $N \in \mb Z$, then
\[k \cdot \mm(a_1, \dots, a_{n-1}; b_1, \dots, b_{n-1}; N - k) \geq \mm(a_1, \dots, a_n; b_1, \dots, b_n; N).\]
\end{lem}
\begin{proof}
Let $y' = (y_1,\dots,y_{n-1})$ be a distribution of $N-k$ balls in the first $n-1$ bins.
Then $y = (y_1,\dots,y_{n-1},k)$ is a distribution of $N$ balls in $n$ bins with the last bin getting $k$ balls.
Therefore,
\[ \mm(a_1,\dots,a_n;b_1,\dots,b_n;N) \leq P(y) = k \cdot P(y'). \]
Since this holds for all such distributions $y'$, we get
\[ \mm(a_1,\dots,a_n;b_1,\dots,b_n;N) \leq k \cdot \mm(a_1,\dots,a_{n-1};b_1,\dots,b_{n-1};N-k). \qedhere\]
\end{proof}

\subsection{Grid Reduction and Condition (D)}
\label{sec:grid}
For any finite grid $A \subset R^n$, evaluation of a polynomial $f \in R[\underline{t}] = R[t_1,\dots,t_n]$ at elements of $A$ gives a ring homomorphism
\[ E_A: R[\underline{t}] \ra R^A, \ f \mapsto (x \in A \mapsto f(x)). \]
Let $I(A)$ be the kernel of $E_A$, i.e., the set of polynomials which vanish identically on $A$.  There are some ``obvious'' elements of $I(A)$, namely
\[ \forall i \in [n], \ \varphi_i = \prod_{x_i \in A_i} \left(t_i - x_i\right). \]
Let $\Phi  = \langle \varphi_1,\dots,\varphi_n \rangle$ be the ideal generated by these elements.  Then $\Phi \subset I(A)$.
\\ \\
We say a polynomial $f \in R[\underline{t}]$ is \textbf{A-reduced} if $\deg_{t_i}(f) < \# A_i$ for all $i \in [n]$.  The $A$-reduced polynomials form an $R$-submodule $\mathcal{R}_{A}$ of $R[\underline{t}]$ which is free of rank $\# A$, and the composite map
\[ \mathcal{R}_A \hookrightarrow R[\underline{t}] \ra R[\underline{t}]/\Phi \]
is an $R$-module isomorphism \cite[Prop.~10]{Clark14}, i.e., every polynomial $f \in R[\underline{t}]$ differs from a unique $A$-reduced polynomial $r_{A}(f)$ by an element of $\Phi$, and we have $E_A(f) = E_A(r_A(f))$.
The polynomial $r_A(f)$ can be computed from $f$ by dividing by $\varphi_1$, then dividing the remainder by $\varphi_2$, and so on. It follows that $\deg r_A(f) \leq \deg f$ and $\deg_{t_i} r_A(f) \leq \deg_{t_i} f$ for all $i \in [n]$.

\begin{thm}[{CATS\footnote{CATS = Chevalley-Alon-Tarsi-Schauz \cite{Chevalley35}, \cite{Alon-Tarsi92}, \cite{Schauz08}.} Lemma, \cite[Thm.~12]{Clark14}}]
The following are equivalent: \\
(i) The finite grid $A$ satisfies Condition (D). \\
(ii) If $f \in \mathcal{R}_A$ and $f(x) = 0$ for all
$x \in A$, then $f = 0$. \\
(iii) We have $\Phi = I(A)$.
\end{thm}

\begin{remark}
These results can be directly used to solve the main problem studied by Alon and F\"{u}redi.
Let $f$ be a polynomial that vanishes on all points of $A$ except the point $x = (x_1, \dots, x_n)$.
Since the polynomial $\prod_{i = 1}^n \prod_{\lambda \in A_i \setminus \{x_i\}} (t_i - \lambda)$ is $A$-reduced and it vanishes everywhere on $A$ except at $x$, it must be equal to $r_A(f)$.
Thus, \[\deg f \geq \deg r_A(f) = \sum_{i=1}^n (\# A_i - 1). \]
Now associate the set of hyperplanes that cover all points of $A$ except one by the product of their corresponding linear polynomials.
\end{remark}

\subsection{Generalized Alon--F\"uredi Implies Alon--F\"uredi}
\label{sec:GAF}
Let $A = \prod_{i=1}^n A_i \subset R^n$ be a finite grid satisfying Condition (D), and for $i \in [n]$ put $a_i = \# A_i$.  Suppose $f \in R[\underline{t}]$ does not vanish identically on $A$.   Let
\[ \mathcal{U}_A(f) = \{x \in A \mid f(x) \neq 0\}. \]
Then the Alon--F\"uredi Theorem is the assertion that
\[ \# \mathcal{U}_A(f) \geq \mm(a_1,\dots,a_n;\sum_{i=1}^n a_i - \deg f). \]
The nonvanishing hypothesis on $f$ is equivalent to
$r_A(f) \neq 0$.  Then $r_A(f)$ satisfies the hypotheses of Theorem \ref{thm:main} with $b_1 = \dots = b_n = 1$.  Since $E_A(f) = E_A(r_A(f))$,
we have $\mathcal{U}_A(f) = \mathcal{U}_A(r_A(f))$, and thus
\[ \# \mathcal{U}_A(f) \geq \mm(a_1,\dots,a_n;1,\dots,1;\sum_{i=1}^n a_i - \deg r_A(f)) \] \[ = \mm(a_1,\dots,a_n;\sum_{i=1}^n a_i - \deg r_A(f)) \geq \mm(a_1,\dots,a_n;\sum_{i=1}^n a_i - \deg f). \]


\section{Proof of the Generalized Alon--F\"{u}redi Theorem}

\subsection{A Preliminary Remark}
If $f$ satisfies the hypotheses of the Generalized Alon--F\"uredi Theorem, then
\[ \deg f \leq \sum_{i=1}^n \deg_{t_i} f \leq \sum_{i=1}^n (a_i- b_i),  \]
so
\[ \sum_{i=1}^n b_i \leq \sum_{i=1}^n a_i - \deg f \leq \sum_{i=1}^n a_i. \]
Thus, whereas the conventional Alon--F\"uredi setup allows the case in which we have too few balls to fill the bins -- in which case the result gives the trivial (but sharp!) bound $\# \mathcal{U}_A(f) \geq 1$ --,
in our setup we do not need to consider this case.

\subsection{Proof of the Generalized Alon--F\"uredi Bound}
\label{sec:proof}
For $i \in [n]$, put $a_i = \# A_i$.
We go by induction on $n$.  \\
\textit{Base Case}: Let $f \in R[t_1]$ be a nonzero polynomial.  Suppose $f$ vanishes precisely at the distinct elements $x_1,\dots,x_k$ of $A_1$.  Dividing $f$ by $t_1-x_1$ shows $f = (t_1-x_1)f_1(t_1)$, and -- since $A_1$ satisfies Condition (D) -- $f_1(x_i) = 0$ for $2 \leq i \leq k$.  Continuing in this way we get $f = \prod_{i=1}^k (t_1-x_i) f_k(t_1)$, and thus $\deg f \geq k$.  So
\[ \# \mathcal{U}_A(f) = a_1 - k \geq a_1 - \deg f, \]
which is the conclusion of the Generalized Alon--F\"uredi Theorem in this case. \\
\textit{Induction Step}: Suppose $n \geq 2$ and the result holds for $n-1$.
Write \[ f(t_1, \dots, t_n) = \sum_{i=0}^{d_n} f_i(t_1,\dots,t_{n-1}) t_n^i, \]
so that $d_n = \deg_{t_n} f$ is the largest index $i$ such that $f_i \neq 0$.
Moreover we have $\deg f_{d_n} \leq \deg f - d_n$ and  for all $i \in [n-1], \deg_{t_i} f_{d_n} \leq \deg_{t_i} f \leq a_i - b_i$. \\
Put $A' = \prod_{i=1}^{n-1} A_i$.  By the induction hypothesis,  we have
\[ \#\m U_{A'}(f_{d_n}) \geq \mm(a_1, \dots, a_{n-1}; b_1, \dots, b_{n-1};
\sum_{i=1}^{n-1}a_i - \deg f_{d_n}) \]
\[ \geq \mm(a_1, \dots, a_{n-1}; b_1, \dots, b_{n-1}; \sum_{i=1}^{n-1}a_i - \deg f + d_n). \]
Let $x' = (x_1,\dots,x_{n-1}) \in \m U_{A'}(f_{d_n})$.
Then $f(x',t_n) = \sum_{i = 0}^{d_n} f_i(x') t_n^i \in R[t_n]$ has degree $d_n \geq 0$ since its leading term $f_{d_n}(x') t^{d_n}$ is non-zero.
Since $A_n$ satisfies Condition (D), $f(x',t_n)$ vanishes at no more than $d_n$ points of $A_n$, so there are at least $a_n-d_n$ elements $x_n \in A_n$ such that $(x',x_n) \in \m U_A(f)$.  Thus
\[
 \# \m U_A(f) \geq (a_n -d_n) \mm(a_1,\dots,a_{n-1};b_1,\dots,b_{n-1}; \sum_{i=1}^{n-1} a_i - \deg f + d_n). \]
Since
\[ \deg f \leq \sum_{i=1}^n \deg_{t_i} f = \sum_{i=1}^{n-1} \deg_{t_i} f + d_n \]
and thus
\[     \sum_{i=1}^{n-1} b_i \leq \sum_{i=1}^{n-1} (a_i- \deg_{t_i} f) \leq \sum_{i=1}^{n-1}a_i - \deg f + d_n
\leq
 \sum_{i=1}^{n-1} a_i, \]
 we may apply Lemma \ref{lem:main} with
 $N = \sum_{i=1}^n a_i - \deg f$ and $k = a_n-d_n$, getting

 \[(a_n -d_n) \mm(a_1,\dots,a_{n-1};b_1,\dots,b_{n-1}; \sum_{i=1}^{n-1} a_i - \deg f + d_n) \geq \]
 \[
 \label{MAININEQ2}
 \mm(a_1,\dots,a_n;b_1,\dots,b_n;\sum_{i=1}^n a_i - \deg f). \]
 We deduce that
 \[ \# \m U_A(f) \geq  \mm(a_1,\dots,a_n;b_1,\dots,b_n;\sum_{i=1}^n a_i - \deg f). \]

\subsection{Sharpness of the Generalized Alon--F\"uredi Bound}
\label{sec:sharp}
For $i \in [n]$, put $a_i = \# A_i$, and
let $d$ be an integer such that $0 \leq d \leq \sum_{i=1}^n (a_i-b_i)$ (see $\S$3.1).
For any distribution $y = (y_1,\dots,y_n)$ of $\sum_{i=1}^n a_i - d$ balls in $n$ bins with
 $b_i \leq y_i \le a_i$, for all $i \in [n]$ choose a subset $S_i \subset A_i$ of cardinality $a_i-y_i$, and put\footnote{An empty product is understood to take the value $1$.}
\[ f(t) = \prod_{i=1}^n \prod_{x_i \in S_i} (t_i-x_i). \]
 Then
\[ \deg f = \sum_{i=1}^n (a_i-y_i) =d, \]
\[\forall i \in [n],~\deg_{t_i} f = a_i-y_i \leq a_i - b_i \]
and
\[ \#\mathcal{U}_A(f) = P(y) = \prod_{i=1}^n y_i. \]
Thus, for all finite grids $A = \prod_{i=1}^n A_i$
satisfying Condition (D) and all permissible values of $\deg_{t_1} f,\dots,\deg_{t_n} f$ and $\deg f$, there are instances of equality in the Generalized Alon--F\"uredi Bound.  The case $b_1 = \dots = b_n = 1$ yields the (known) sharpness of
the Alon--F\"uredi Bound.

\subsection{An equivalent formulation}
\label{sec:Petrov}
Let us say that a polynomial $f \in R[\underline{t}]$ is \textbf{polylinear} (resp. \textbf{simple polylinear}) if it is a product of factors (resp. distinct factors) of the form $t_i - x$ for $x \in R$. 
Petrov has observed (personal communication) that the Generalized Alon--F\"uredi Theorem is equivalent to the statement that for any
nonzero $A$-reduced polynomial $f \in R[\underline{t}]$, there is a simple polylinear polynomial $g \in R[\underline{t}]$ with $\deg_{t_i} f = \deg_{t_i} g$ for all $i \in [n]$, $\deg f = \deg g$ and such that $\# Z_A(f) \leq \# Z_A(g)$.  Thus it is possible to formulate the result without reference to balls in prefilled bins.  However, as we will see, having the result in this form is useful for applications.


\section{Connections with the Schwartz--Zippel Lemma}
\label{sec:SZ_main}
\subsection{Schwartz--Zippel Lemma}
\label{sec:SZ}
The material in this section is motivated by a blog post of Lipton \cite{Lipton09} which discusses the history of the ``Schwartz--Zippel Lemma''.  We will further weigh in on the history of this circle of results, discuss various improvements and give the connection to the Alon--F\"uredi Theorem.

\begin{thm}[Schwartz--Zippel Lemma] Let $R$ be a domain and let $S \subset R$ be finite and nonempty.  Let $f \in R[\underline{t}] = R[t_1,\dots,t_n]$ be a nonzero polynomial.  Then
\begin{equation}
\label{SZEQ1}
 \# Z_{S^n}(f) \leq (\deg f) (\# S)^{n-1}.
 \end{equation}
\end{thm}
\begin{proof}
Let $s = \# S$.  The statement is equivalent to
\[ \# \mathcal{U}_{S^n}(f) \geq s^{n-1}(s - \deg f). \]
If $\deg f \geq s$, then (\ref{SZEQ1}) asserts
that $f$ has no more zeros in $S^n$ than the size of
$S^n$: true.  So the nontrivial case is $\deg f < s$.  Then $f$ is $S^n$-reduced, so
\[ \# \mathcal{U}_{S^n}(f) \geq \mm(s,\dots,s;n s - \deg f) = s^{n-1}(s - \deg f) \]
by Alon--F\"uredi and because the greedy distribution is $(s,\dots,s,s - \deg f)$.
\end{proof}
\noindent
The case of the Schwartz--Zippel Lemma in which $R = S = \F_q$ is due to Ore \cite{Ore22}.  Thus the Schwartz--Zippel Lemma may be viewed as a ``Restricted Variable Ore Theorem'', although it is not the most general result along those lines.  In fact, the same argument establishes the following.

\begin{thm}[Generalized Schwartz--Zippel Lemma]
\label{LOWDEGREEAFTHM}
Let $A = \prod_{i=1}^n A_i \subset R^n$ be a finite grid satisfying Condition (D), and suppose $\# A_1 \geq \dots \geq \# A_n$.  Let $f \in R[\underline{t}] = R[t_1,\dots,t_n]$ be a nonzero polynomial.  Then
\[ \# Z_A(f) \leq (\deg f) \prod_{i=1}^{n-1} \# A_i. \]
\end{thm}
\noindent
Indeed, the nontrivial case of this result is precisely the case $\#A_1 \geq \dots \geq \#A_n > \deg f$ of Alon--F\"uredi, and the greedy distribution is $(\#A_1,\dots, \#A_{n-1}, \#A_n-\deg f)$.

\subsection{Schwartz's Theorem}
\label{sec:Schwartz}
The Schwartz--Zippel Lemma appears in Schwartz's 1980 paper as a corollary of a more general upper bound on zeros of a polynomial over a domain \cite[Cor. 1]{Schwartz80}.  We give a version over an arbitrary ring.

\begin{thm}[{Schwartz Theorem \cite[Lemma 1]{Schwartz80}}]
\label{SCHWARTZ}
Let  $f = f_n \in R[t_1,\dots,t_n]$ be a
nonzero polynomial and let $d_n = \deg_{t_n} f_n$.
Let $f_{n-1} \in R[t_1,\dots,t_{n-1}]$ be the coefficient of $t_n^{d_n}$
in $f_n$.  Let $d_{n-1} = \deg_{t_{n-1}} f_{n-1}$, and let $f_{n-2} \in R[t_1,\dots,t_{n-2}]$
be the coefficient of $t_{n-1}^{d_{n-1}}$ in $f_{n-1}$.  Continuing in this
manner we define for all $1 \leq i \leq n$ a polynomial
$f_i \in R[t_1,\dots,t_i]$ with $\deg_{t_i} f_i = d_i$.
Let $A = \prod_{i=1}^n A_i$ be a finite grid
satisfying Condition (D).  Then
\[ \# Z_A(f) \leq  \# A \sum_{i=1}^n \frac{d_i}{\# A_i}. \]
\end{thm}
\begin{proof}
For $i \in [n]$, put $a_i = \# A_i$.  We go by induction on $n$.  The base case is the same as that of Theorem \ref{thm:main}: essentially the root-factor phenomenon of high school algebra, used with some care because $R$ need not be a domain.  Inductively
we suppose the result holds for polynomials in $n-1$ variables and in particular for $f_{n-1} \in R[t_1,\dots,t_{n-1}]$
and $A' = \prod_{i=1}^{n-1} A_i$.  Let $x' = (x_1,\dots,x_{n-1}) \in A'$.
If $f_{n-1}(x') = 0$, it may be the case that $f_n(x', x_n) = 0$ for all
$x_n \in A_n$.  But if not, then $f_n(x', t_n) \in R[t_n]$ has at most
$d_n$ zeros in $A_n$.  Thus the number of zeros of $f = f_n$ in
$A$ is at most
\[ \# A_n \cdot \# A' \left( \sum_{i=1}^{n-1} \frac{d_i}{a_i} \right) +
d_n \# A' = \# A \sum_{i=1}^n \frac{d_i}{a_i}.  \qedhere \]
\end{proof}
\noindent

\begin{prop}
Theorem \ref{SCHWARTZ} implies Theorem \ref{LOWDEGREEAFTHM}.
\end{prop}
\begin{proof} The coefficient of $t_1^{d_1} \cdots t_n^{d_n}$ in $f$ is nonzero, so $\sum_{i=1}^n d_i \leq \deg f$, and thus
\[ \# Z_A(f) \leq \# A \sum_{i=1}^n \frac{d_i}{\# A_i} \leq
(\# A_1 \cdots \# A_{n-1}) \sum_{i=1}^n d_i \leq (\deg f) \prod_{i=1}^{n-1} \# A_i. \qedhere \]
\end{proof}


\subsection{DeMillo--Lipton and Zippel}
\label{sec:DMLZ}
The following result was proved by DeMillo and Lipton in \cite{DeMillo--Lipton78} and then independently by Zippel \cite{Zippel79}.

\begin{thm}[DeMillo--Lipton--Zippel Theorem]
\label{DMLZ}
Let $R$ be a domain, let $f \in R[\underline{t}] = R[t_1,\dots,t_n]$ be a nonzero polynomial, and let  $d \in \Z^+$ be such that $\deg_{t_i} f \leq d$ for all $i \in [n]$. Let $S \subset R$ be a nonempty set with more than $d$ elements.  Then
\[ \#Z_{S^n}(f) \leq (\# S)^n - (\# S-d)^n. \]
\end{thm}
\begin{proof}
Put $s = \# S$.  We go by induction on $n$, and $n = 1$ case is by now familiar.  Assume
the result for $n-1$.  Since $\deg_{t_n} f \leq d$, we have
\[ \# \{x_n \in S \mid f(t_1,\dots,t_{n-1},x_n) = 0\} \leq d, \]so there are at least $s-d$ values of
$x_n$ such that $g = f(t_1,\dots,t_{n-1},x_n)$ is a nonzero
polynomial.  By induction, $g$ has at most $s^{n-1} - (s-d)^{n-1}$
zeros on $S^{n-1}$.  So
\[ \# Z_{S^n}(f) \leq d s^{n-1} + (s-d)(s^{n-1}-(s-d)^{n-1})\]
\[ = ds^{n-1} + s^n -ds^{n-1} - (s-d)^n = s^n - (s-d)^n. \qedhere \]
\end{proof}
\noindent
Just like the Schwartz--Zippel Lemma, a stronger version of the DeMillo--Lipton--Zippel Theorem can be proved with essentially the same argument.  We leave the proof -- or rather this proof -- to the reader.

\begin{thm}[Generalized DeMillo--Lipton--Zippel Theorem]
\label{GENDMLZ}
Let $R$ be a ring, $f \in R[t_1,\dots,t_n]$ a nonzero polynomial, and for $i \in [n]$ put $d_i = \deg_{t_i} f$. Let $A = \prod_{i=1}^n A_i$ be a finite grid
satisfying Condition (D).  We suppose that
$1 \leq d_i < a_i$ for all $i \in [n]$.  Then
\[ \# \mathcal{U}_A(f) \geq \prod_{i=1}^n (\# A_i- d_i). \]
\end{thm}
\noindent
Now for a somewhat unsettling remark: the DeMillo--Lipton--Zippel Theorem does not imply the Schwartz--Zippel Lemma nor is it implied by any of Schwartz's results!

\begin{example}
\label{EXAMPLE2}
Let $S$ be a finite subset of $R$ containing $0$, satisfying Condition (D), and of size $s \geq 3$.
Let $f = t_1t_2 \in R[t_1, t_2]$.  Then we have
\[ \# Z_{S^2}(f) = 2s-1. \]
DeMillo--Lipton--Zippel gives
\[ \# Z_{S^2}(f) \leq s^2 - (s-1)^2 = 2s-1. \]
Schwartz's Theorem gives
\[ \# Z_{S^2}(f) \leq s^2 \left(\frac{1}{s} + \frac{1}{s} \right) = 2s. \]
The Alon--F\"uredi Theorem gives
\[ \# Z_{S^2}(f) \leq s^2 - \mm(s,s;2s-2) = s^2 - s(s-2) = 2s. \]
Thus neither Alon--F\"uredi nor Schwartz implies DeMillo--Lipton--Zippel.
For the other direction, take $f = t_1 + t_2$.
DeMillo--Lipton--Zippel gives  $\#Z_{S^2}(f) \leq s^2 - (s - 1)^2 = 2s - 1$, while the other results give $\#Z_{S^2}(f) \leq s$.
\end{example}

\noindent
But we can still relate Schwartz--Zippel and DeMillo--Lipton--Zippel as follows.

\begin{prop}The Generalized Alon--F\"uredi Theorem implies the Generalized DeMillo--Lipton--Zippel Theorem.
\end{prop}
\begin{proof} For $i \in [n]$, put $a_i = \# A_i$ and $b_i =
a_i - d_i$, so $1 \leq b_i \leq a_i$ for all $i \in [n]$.  Moreover $\deg f \leq \sum_{i=1}^n d_i$, so
the Generalized-Alon--F\"uredi Theorem gives
\[ \uu_A \geq \mm(a_1,\dots,a_n;b_1,\dots,b_n;\sum_{i=1}^n a_i- \deg f) \geq \mm(a_1\dots,a_n;b_1,\dots,b_n;\sum_{i=1}^n (a_i - d_i)) \] \[ =
\mm(a_1,\dots,a_n;b_1,\dots,b_n;\sum_{i=1}^n b_i) = \prod_{i=1}^n b_i =
\prod_{i=1}^n (a_i-d_i).  \qedhere\]
\end{proof}
\noindent
Generalized DeMillo--Lipton--Zippel is equivalent to the case $\deg f = \sum_{i=1}^n \deg_{t_i} f$ of Generalized Alon--F\"uredi.  In particular, the bound is sharp in every case.

\section{Connections with Coding Theory}
\label{sec:coding_theory}
\noindent
In this section we make use of some terminology (only) from coding theory.  Definitions can be found in e.g. \cite[Ch. 3]{vanLint}.
\\ \\
Consider the polynomial ring $\F_q[\underline{t}] = \F_q[t_1,\dots,t_n]$.  Since $\F_q$ is finite, we can take $\F_q^n$ itself as a finite grid, and in fact many aspects of the theory presented here were worked out in this case in the early part of the 20th century.  In particular, we say a polynomial is \textbf{reduced} if it is $\F_q^n$-reduced, and this
notion was introduced by Chevalley in his seminal work \cite{Chevalley35} on polynomial systems of low degree.  We denote the the set of reduced polynomials by $\mathcal{P}(n,q)$; it is an $\F_q$-vector space of dimension $q^n$.  The evaluation map gives an $\F_q$-linear isomorphism
\[ E: \mathcal{P}(n,q) \ra \F_q^{\F_q^n}, \ f \mapsto (x \in \F_q^n \mapsto f(x)). \]
Fixing an ordering $\alpha_1,\dots,\alpha_{q^n}$ of $\F_q^n$, this isomorphism allows us to identify each $f \in \mathcal{P}(n,q)$ with its \textbf{value table} $(f(\alpha_1),\dots,f(\alpha_{q^n}))$.  For $d \in \N$ we denote by $\mathcal{P}_d(n,q)$ the set of all reduced polynomials of degree at most $d$. \\

\begin{mydef}
The set of all value tables of all polynomials in $\m P_d(n, q)$ is called the $d$-th order generalized Reed--Muller code of length $q^n$, denoted by ${\rm GRM}_d(n, q)$.
\end{mydef}
\noindent \\
For $q = 2$, these codes were introduced and studied by Muller \cite{Muller54} and Reed \cite{Reed54}.
An explicit formula for minimum distance of the generalized Reed--Muller codes was given by Kasami, Lin and Peterson \cite{Kasami--Lin--Peterson68}, which we will recover using Alon--F\"uredi.
A systematic study of these codes in terms of the polynomial formulation was conducted by Delsarte,  Goethals and MacWilliams in \cite{Delsarte-Goethals-MacWilliams70}, where they also classified all the minimum weight codewords.

\begin{thm}[Kasami--Lin--Peterson]
The minimum weight of the $d$-th order Generalized Reed--Muller code ${\rm GRM}_d(n, q)$ is equal to $(q-b)q^{n-a-1}$ where $d = a(q-1) + b$ with $0 < b \leq q-1$.
\end{thm}
\begin{proof}
The minimum weight of ${\rm GRM}_d(n,q)$ is equal
to the least number of nonzero values taken by a
nonzero reduced polynomial of degree at most $d$, which by Alon--F\"uredi is $\mm(q,\dots,q;nq-d)$.  Moreover we have \[(nq - d) - n = n(q-1) - a(q-1) - b = (n - a - 1)(q-1) + q - 1 - b, \] and  \[0 \leq q -1 - b < q - 1, \] so by Lemma \ref{lem:FS} we have
\[ \mm(q,\dots,q;nq-d) = (q-b) q^{n-a-1}. \qedhere \]
\end{proof}
\noindent
The Generalized Alon--F\"uredi Theorem can also be stated in terms of coding theory.  Let $A = \prod_{i=1}^n A_i$ be a finite grid in a ring $R^n$ satisfying Condition (D), with $a_i = \# A_i$ for $i \in [n]$.  Given positive integers $b_i \leq a_i$ for all $i \in [n]$, and a natural number $d \leq \sum_{i=1}^n (a_i-b_i)$, we define the \textbf{generalized affine grid code}
${\rm GAGC}_d(A;b_1,\dots,b_n)$ as the set of value tables of all polynomials $f \in R[\underline{t]}$ with $\deg_{t_i} f \leq a_i-b_i$ for all $i \in [n]$ and $\deg f \leq d$ evaluated on $A$.  We put
\[ {\rm AGC}_d(A) = {\rm GAGC}_d(A;1,\dots,1) \]
and speak of \textbf{affine grid codes}.  Then:

\begin{thm} The minimum weight of ${\rm GAGC}_d(A;b_1,\dots,b_n)$ is \\ $\mm(a_1,\dots,a_n;b_1,\dots,b_n;\sum_{i=1}^n a_i - d)$.
\end{thm}
\noindent
Affine grid codes were studied in \cite{LRMV14} (under the name of Affine Cartesian Codes) where  they proved the following:

\begin{thm}[{\cite[Thm. 3.8]{LRMV14}}]
\label{LRMVTHM}
Let $F$ be a field and $A = \prod_{i=1}^n A_i \subset F^n$ a finite grid with $\# A_1 \geq \dots \geq \# A_n \geq 1$.
Then the minimum weight of ${\rm AGC}_d(A)$ is \\
$\begin{cases} \# A_1 \cdots \# A_{k-1} (\#A_k - \ell) & \text{ if } d \leq \sum_{i=1}^n(\# A_i-1) -1 \\
1 & \text{ if } d \geq \sum_{i=1}^n (\# A_i-1) \end{cases}$, \\
where $k,\ell \in \Z$ are such that $d = \sum_{i=k+1}^n (\# A_i-1) + \ell$, $k \in [n]$ and
 $\ell \in [\# A_k-1]$.

\end{thm}
\begin{proof}
The minimum weight of ${\rm AGC}_d(A)$ is $\mm(\#A_1,\dots,\#A_n;\sum_{i=1}^n \# A_i - d)$. So the result follows from Lemma \ref{lem:FS}, as the greedy distribution of \[\sum_{i = 1}^n \# A_i - d = \sum_{i = 1}^{k-1} (\#A_i - 1) + (\# A_k - 1 - \ell) + n\] balls is $(\# A_1, \dots, \#A_{k-1}, \# A_k - \ell, 1, \dots, 1)$.
\end{proof}
\begin{remark}
\begin{enumerate}[$(a)$]
\item The paper \cite{LRMV14} makes no mention of Alon--F\"uredi.  Their proof of Theorem \ref{LRMVTHM} is self-contained and thus gives a proof of Alon--F\"uredi with the balls in bins constant replaced by its explicit value $P(y_G)$.  On the other hand it is longer than the other proofs of Alon--F\"uredi appearing in the literature. 
\item Our proof of Theorem \ref{LRMVTHM} works for a grid $A \subset R^n$ satisfying Condition (D).  
\item When $b_1 \geq \dots \geq b_n$, the greedy algorithm computes $\mm(a_1,\dots,a_n;b_1,\dots,b_n;N)$ and we could give a similarly explicit description of ${\rm GAGC}_d(A_1;b_1,\dots,b_n)$. 
\item While (binary) Reed--Muller codes are mentioned in \cite{Alon--Furedi93} under Corollary 1, the connection between  Alon--F\"uredi Theorem and generalised Reed--Muller codes is not explored. 
\end{enumerate}
\end{remark}

\section{Applications to Finite Geometry}
\label{sec:finite_geometry}

\subsection{Partial Coverings of Grids by Hyperplanes}
By a \textbf{hyperplane} in $R^n$ we mean a polynomial $H = c_1 t_1 + \dots + c_n t_n + r \in R[\underline{t}]$ for which at least one $c_i$ is not a zero-divisor.  (Referring to the polynomial itself rather than its zero locus in $R^n$ will make the discussion cleaner.)  A family $\mathcal{H} = \{H_i\}_{i=1}^d$ \textbf{covers} $x \in R^n$ if $H_i(x) = 0$ for some
$i \in [n]$; $\mathcal{H}$ \textbf{covers} a subset $S \subset R^n$ if it covers every point of $S$, and $\mathcal{H}$ \textbf{partially covers} $S$ otherwise. For a family $\mathcal{H} = \{H_i\}_{i=1}^d$ of hyperplane in $R^n$, put
\[ f_{\mathcal{H}} = \prod_{i=1}^d H_i. \]
Thus $f_{\mathcal{H}}$ is a polynomial of degree $d$.  If $\mathcal{H}$ covers $A$, then $f$ vanishes identically on $A$.  If $R$ is a domain the converse holds, and thus $\mathcal{H}$ covers $A$ iff $f_{\mathcal{H}} \in \langle \varphi_1,\dots,\varphi_n \rangle$.
We now revisit the original combinatorial problem studied by Alon and F\"uredi, which is part c) of the following theorem.  However, our proof is via Theorem \ref{thm:AF} instead of the approach used in \cite{Alon--Furedi93}.


\begin{thm}
\label{thm:hyperplanes}
Let $R$ be a domain, let $A = \prod_{i=1}^n A_i \subset R^n$ be a finite grid, and let $\mathcal{H} = \{H_i\}_{i = 1}^d$ be family of
hyperplanes in $R^n$. 
\begin{enumerate}[$(a)$]
\item If $\mathcal{H}$ partially covers $A$, then $\mathcal{H}$ fails to cover at least \\ $\mm(\# A_1,\dots,\# A_n;\sum_{i=1}^n \# A_i - d)$ points of $A$.  
\item For all $d \in \Z^+$, there is a family of hyperplane $\{H_i = t_{j_i}-x_i\}_{i=1}^d$ with
$j_i \in [n]$ and $x_i \in A_i$ which covers all but exactly 
$\mm(\#A_1,\dots,\# A_n;\sum_{i=1}^n \# A_i - d)$ points of $A$. 
\item If $\mathcal{H}$ covers all but exactly one point of $A$, then $d \geq \sum_{i=1}^n (\# A_i-1)$.
\end{enumerate}
\end{thm}
\begin{proof}
$(a)$ As above, $\mathcal{H}$ covers $x \in R^n$ iff $f_{\mathcal{H}}(x) = 0$.  Apply Alon--F\"uredi.  \\
$(b)$ The sharpness construction of $\S$\ref{sec:sharp} is precisely of this form. \\
$(c)$ If $\mathcal{H}$ covers all points of $A$ except one, then \[ 1 = \#\mathcal{U}_A(H_1 \cdots H_d) \geq \mm(\# A_1,\dots,\# A_n;\sum_{i=1}^n \# A_i - d), \]
so Lemma \ref{LEMMA2.2} gives $\sum_{i=1}^n \# A_i - d \leq n$, i.e. $d \geq \sum_{i=1}^n (\#A_i-1)$.
\end{proof}
\noindent
We complement Theorem \ref{thm:hyperplanes} by computing the minimum cardinality of a hyperplane covering of a finite grid (not necessarily specifying Condition (D)) over a ring $R$.

\begin{thm}
\label{NEWPOTUTHM}
Let $A = \prod_{i=1}^n A_i \subset R^n$ be a finite grid, and let $\mathcal{H} = \{H_i\}_{i=1}^d$ be a hyperplane covering of $A$.  Then $d \geq \min \# A_i$.
\end{thm}
\begin{proof}
 First we observe that if $A$ satisfies Condition (D) then the result is almost immediate: going by contraposition, if $d \leq \# A_i-1$ for all $i \in [n]$ then $f_{\mathcal{H}}$ is nonzero and $A$-reduced, so it cannot vanish identically on $A$. \\ \indent
Now we give a non-polynomial method proof in the general case.  Without loss of generality assume $\# A_1 \geq \dots \geq \dots \geq \# A_n$.  We {\sc claim} that any hyperplane $H = \sum_{i=1}^n c_i t_i + g$ covers at most $\prod_{i=1}^{n-1} \# A_i$ points of $A$: this suffices, for then $d \geq \# A_n$.
\\
{\sc proof of claim:} Fix $i \in [n]$ such that $c_i$ is not a zero-divisor in $R$.  Let $\pi: R^n \ra R^{n-1}$ be the projection
$(x_1,\dots,x_n) \mapsto (x_1,\dots,x_{i-1},x_{i+1},\dots,x_n)$.  Then
\[ A = \coprod_{x' = (x_1,\dots,x_{i-1},x_{i+1},\dots,x_n) \in \pi(A)} \{x_1 \} \times \dots \times \{x_{i-1}\} \times A_i \times \{x_{i+1}\} \times \dots \{x_n\} \]
is a partition of $A$ into $\# \pi(A) = \prod_{j \neq i} \# A_j$ nonempty subsets, each one of which meets $H$ in at most one point.  So
\[ \# \left( Z(H) \cap A \right) \leq \prod_{j \neq i} \# A_j \leq \prod_{i=1}^{n-1} \# A_i. \qedhere \]
\end{proof}

\begin{conj}
\label{CONJ1}
Let $R$ be a ring, and let $A_1,\dots,A_n \subset R$ be nonempty (but possibly infinite).  Let $\mathcal{H} = \{H_j\}_{j \in J}$ be a covering of the grid $A = \prod_{i=1}^n A_i$ by hyperplanes.  Then $\# J \geq \min_{i=1}^n \# A_i$.
\end{conj}

\begin{remark}
\begin{enumerate}[$(a)$]
\item For $i \in [n]$, let $B_i \subset A_i \subset R$.
Then we need at least as many hyperplanes to cover $\prod_{i=1}^n A_i$ as we do to cover $\prod_{i=1}^n B_i$.  Together with Theorem \ref{NEWPOTUTHM} it follows that in the setting of Conjecture \ref{CONJ1} we need at least $\min(\# A_i,\aleph_0)$ hyperplanes. Thus Conjecture \ref{CONJ1} holds when $R$ is countable.  
\item When $R$ is a field and $A = R^n$, Conjecture
\ref{CONJ1} is a case of \cite[Thm. 3]{Clark12}.
\end{enumerate}
\end{remark}



\subsection{Partial Covers and Blocking Sets in Finite Geometries}
\label{sec:partial_blocking}
The same ideas can be used to prove old and new results about Desarguesian projective and affine spaces over finite fields.
\\ \\
Let $\mathrm{PG}(n,q)$ denote the $n$-dimensional projective space over $\mb F_q$ (the same object would in some other circles be denoted by $\mathbb{P}^n(\F_q)$) and let $\mathrm{AG}(n,q)$ denote the $n$-dimensional affine space over $\mb F_q$ (resp. $\mathbb{A}^n(\F_q)$).  The set $\mathrm{AG}(n,q)$ comes equipped with a sharply transitive action of the additive group of $\F_q^n$ and thus a choice of a point $x \in \mathrm{AG}(n,q)$ induces an isomorphism $\mathrm{AG}(n,q) \cong \F_q^n$.  We will make such identifications without further comment.
\\ \\
A \textbf{partial cover} of $\mathrm{PG}(n,q)$ is a set of hyperplanes that do not cover all the points.  The points missed by a partial cover are called  \textbf{holes}.

\begin{thm}
\label{PARTIALCOVERTHM}
Let $\m H$ be a partial cover of $\mathrm{PG}(n,q)$ of size $k \in \Z^+$.
Then $\m H$ has at least $\mm(q, \dots, q; nq - k + 1)$ holes.
\end{thm}
\begin{proof}
Let $H \in \mathcal{H}$.  Then $\mathrm{PG}(n,q) \setminus H \cong  \mathrm{AG}(n,q)$ so $\mathcal{H} \setminus H$ is a partial cover of $\F_q^n$ by $k-1$ hyperplanes.  As above, there are at least $\mm(q, \dots,q; nq - (k-1))$ points not covered by $\m H$.
\end{proof}

\begin{cor}
\label{PARTIALCOVERCOR}
If $0 \leq a < q$, a partial cover of $\mathrm{PG}(n,q)$ of size $q+a$ has at least $q^{n-1} - aq^{n-2}$ holes.
\end{cor}
\begin{proof}
By Theorem \ref{PARTIALCOVERTHM} there are at least $\mm(q,\dots,q;(n-1)q-a + 1)$ holes.  Since $0 \leq a < q$, the greedy distribution is $(q,\dots,q,q-a,1)$, and the result follows.
\end{proof}
\noindent
Dodunekov, Storme and Van de Voorde have shown that a partial cover of $\mathrm{PG}(n,q)$ of size $q+a$ has at least $q^{n-1} - a q^{n-2}$ holes if $0 \leq a < \frac{q-2}{3}$ \cite[Thm. 17]{DSV10}.  Corollary \ref{PARTIALCOVERCOR} gives an improvement in that the restriction on $a$ is relaxed.
They also show that if $a < \frac{q - 2}{3}$ and the number of holes are at most $q^{n-1}$, then they are all contained in a single hyperplane.
We cannot make any such conclusions from our arguments.
\\ \\
Projective duality yields a dual form of Theorem \ref{PARTIALCOVERTHM}: $k$ points in $\mathrm{PG}(n,q)$ which do not meet all hyperplanes must miss at least $\mm(q, \dots, q; nq - k + 1)$ of them.  Thus:

\begin{thm}
\label{thm:affine}
Let $S$ be a set of $k$ points in $\mathrm{AG}(n,q)$.
Then there are at least $\mm(q, \dots, q; nq - k + 1) - 1$ hyperplanes of $\mathrm{AG}(n,q)$ which do not meet $S$.
\end{thm}
\begin{proof}
Add a hyperplane at infinity to get to the setting of $\mathrm{PG}(n,q)$ and then apply the dual form of Theorem \ref{PARTIALCOVERTHM}.
\end{proof}
\noindent
The general problem of the number of linear subspaces missed by a given set of points in $\mathrm{PG}(n,q)$ is studied by Metsch in \cite{Metsch06}.
We wish to note that Theorem \ref{thm:affine} gives the same bounds as Part $(a)$ of \cite[Theorem 1.2]{Metsch06} for the specific case when the linear subspaces are hyperplanes.
\\ \\
A \textbf{blocking set} in $\mathrm{AG}(n,q)$ or $\mathrm{PG}(n,q)$ is a set of points that meets every hyperplane.  The union of the coordinate axes in $\F_q^n$ yields a blocking set in $\mathrm{AG}(n,q)$ of size $n(q-1) + 1$.  Doyen conjectured in a 1976 Oberwolfach lecture that $n(q-1)+1$ is the least possible size of a blocking set in $\mathrm{AG}(n,q)$.  A year later two independent proofs appeared, by Jamison \cite{Jamison77}, and then a (simpler) proof by Brouwer and Schrijver \cite{Brouwer--Schrijver78}.  We are in a position to give another proof.

\begin{cor}[Jamison--Brouwer--Schrijver]
The minimum size of a blocking set in $\mathrm{AG}(n,q)$ is $n(q-1)+1$.
\end{cor}
\begin{proof}
Let $B \subset \mathrm{AG}(n,q)$ be a blocking set of cardinality at most $n(q-1)$.  By Theorem \ref{thm:affine} and Lemma \ref{LEMMA2.2} there are at least \[\mm(q, \dots, q; nq - n(q-1) + 1) - 1 = \mm(q, \dots, q; n + 1) - 1 \geq 1 \] hyperplanes which do not meet $B$.
\end{proof}
\noindent
Turning now to $\mathrm{PG}(n,q)$, every line is a blocking set.  But classifying blocking sets that do not contain any line is one of the major open problems in finite geometry. For a survey on blocking sets in finite projective spaces, see \cite[Chapter 3]{DeBeule-Storme11}.
\\ \\
If $B \subset \mathrm{PG}(n,q)$, $x \in B$ and $H$ is a hyperplane in $\mathrm{PG}(n,q)$, then \textbf{$H$ is a tangent to B through x} if $H \cap B = \{x\}$.  An \textbf{essential point} of a blocking set $B$ in $\mathrm{PG}(n,q)$ is a point $x$ such that $B \setminus \{x\}$ is not a blocking set.  A point $x$ of $B$ is essential if and only if there is a tangent hyperplane to $B$ through $x$.


\begin{thm}
\label{THM7.5}
Let $B$ be a blocking set in $\mathrm{PG}(n,q)$ and $x$ an essential point of $B$.
There are at least $\mm(q, \dots, q; nq - \# B + 2)$ tangent hyperplanes to $B$ through $x$.
\end{thm}
\begin{proof}
Let $H$ be a tangent hyperplane to $B$ through $x$.  Then $B' = B \setminus \{x\} \subset \mathrm{PG}(n,q) \setminus H \cong \mathrm{AG}(n,q)$.  By Theorem \ref{thm:affine} there are at least $\mm(q, \dots, q; nq - \#B + 2) - 1$ hyperplanes in $\mathrm{AG}(n,q)$ that do not meet $B'$.
Since $B$ is a blocking set all of these hyperplanes, when seen in $\mathrm{PG}(n,q)$, must meet $x$.  Thus there are at least $\mm(q, \dots, q; nq - \#B + 2)$ tangent hyperplanes to $B$ through $x$.
\end{proof}

\begin{cor}[Blokhuis--Brouwer \cite{Blokhuis--Brouwer86}]
Let $B$ be a blocking set in $\mathrm{PG}(2, q)$ of size $2q-s$.  There are at least $s + 1$ tangent
lines through each essential point of $B$.
\end{cor}
\begin{proof}
By Theorem \ref{THM7.5}, each essential point of
$B$ has at least \[\mm(q,q;2q-(2q-s)+2) = \mm(q,q;s+2)\] tangent lines.  Since $\# B = 2q-s < q^2 + q+ 1 = \# \mathrm{PG}(2,q)$, there is $x \in \mathrm{PG}(2,q) \setminus B$.  There are $q+1$ lines through $x$, so $2q-s = \# B \geq q+1$.  Thus $s+1 \leq q$, so the greedy distribution is $(s+1,1)$ and $\mm(q, q; s + 2) = s+1$.
\end{proof}

\begin{cor}{\cite[Theorem $7$]{DSV10}}
If $0 \leq a < q$, there are at least $q^{n-1} - aq^{n-2}$ tangent hyperplanes through each essential point of a blocking set of size $q+a+1$ in $\mathrm{PG}(n,q)$.
\end{cor}
\begin{proof}
By Theorem \ref{THM7.5} and the proof of Corollary \ref{PARTIALCOVERCOR}, each essential point of
$B$ has at least $\mm(q,\dots,q;nq-(q+a+1)+2) = q^{n-1}-aq^{n-2}$ tangent hyperplanes.
\end{proof}

\section{Multiplicity Enhancements}
\label{sec:mult}
\noindent
That one can assign to a zero of a polynomial a positive integer called \textbf{multiplicity} is a familiar concept in the univariate case.  The definition of the multiplicity $m(f,x)$ of a multivariate polynomial $f \in R[\underline{t}]$ at a point $x \in R^n$ (see $\S$7.2) may be less familiar, but the concept is no less useful.  All of the main results considered thus far are upper bounds on $\# Z_A(f)$, the number of zeros of a polynomial $f$ on a finite grid.  By a \textbf{multiplicity enhancement} we mean the replacement of $\# Z_A(f)$ by $\sum_{x \in A} m(f,x)$ in such an upper bound.  The prototypical example: for a nonzero univariate polynomial $f$ over a field $F$ we have $\sum_{x \in F} m(f,x) \leq \deg f$.
\\ \\
Recently, multiplicity enhancements have become part of the polynomial method toolkit. In \cite{DKSS13} Dvir, Kopparty, Saraf and Sudan gave a multiplicity enhancement of the Schwartz--Zippel Lemma.  This was a true breakthrough with important applications in both combinatorics and theoretical computer science.  In $\S$4 we saw that the original work of Schwartz, DeMillo--Lipton and Zippel consists of more than the Schwartz--Zippel Lemma and gave some extensions of this work, in particular working over an arbitrary ring.  So it is natural to consider multiplicity enhancements of these results.  We do so here, giving a multiplicity enhancement of
Theorem \ref{SCHWARTZ} and thus also of Theorem \ref{LOWDEGREEAFTHM}.  On the other hand the Alon--F\"uredi Theorem does not allow for a multiplicity enhancement (at least not in the precise sense described above), as we will see in Example \ref{EXAMPLE3}.
\\ \\
This is a situation where working over a ring under Condition (D) makes things a bit harder.
Lemma \ref{SOUPEDROOTFACTORLEMMA} pushes through the single variable root-factor phenomenon under Condition (D).
\\ \\
In places our treatment closely follows that of \cite{DKSS13}.  We need to set things up over a ring, whereas they work over a field.  Nevertheless, their work carries over verbatim much of the time, and when this is the case we state the result in the form we need it, cite the analogous result in \cite{DKSS13} and omit the proof.


\subsection{Hasse Derivatives}
\label{sec:Hasse}
Let $R[\underline{t}] = R[t_1,\dots,t_n]$.  For $I = (i_1,\dots,i_n) \in \N^n$, put
\[ \underline{t}^I = t_1^{i_1} \cdots t_n^{i_n}\]
and $|I| = \sum_{j = 1}^n i_j = \deg \underline{t}^I$.
Thus, $\{\underline{t}^I\}_{I \in \N^n}$ is an $R$-basis for
$R[\underline{t}]$.
We put
\[ {I \choose J} = \prod_{k=1}^n {i_k \choose j_k}, \]
taking ${i \choose j} = 0$ if $j > i$.
\\ \\
For $J \in \N^n$, let $D^J: R[\underline{t}] \ra R[\underline{t}]$ be the unique $R$-linear map
such that
\[ D^J(\underline{t}^I) = {I \choose J} \underline{t}^{I-J}. \]
We have $D^J(\underline{t}^I) = 0$, unless $J \leq I$.
Repeated application of the identity
\[ t^n = (t-x+x)^n = \sum_{j=0}^n {n \choose j} x^{n-j} (t-x)^j \]
leads to the Taylor expansion: for $f \in R[\underline{t}]$
and $x \in R^n$,
\begin{equation}
\label{TAYLOREXP}
 f(\underline{t}) = \sum_J D^J(f)(x) (\underline{t}-x)^J.
\end{equation}
Applying the automorphism $\underline{t} \mapsto \underline{t}+x$ gives
the alternate form
\[ f(\underline{t}+x) = \sum_J D^J(f)(x) \underline{t}^J. \]
These $D^J(f)$ were defined in \cite{Hasse36} and are now called \textbf{Hasse derivatives}.

\begin{prop}[{\cite[Prop. 2.3]{DKSS13}}]
Let $f \in R[\underline{t}]$, and let $I,J \in \N^n$.  
\begin{enumerate}[$(a)$]
\item If $f$ is homogeneous of degree $d$ and $D^I(f)$ is non-zero, then $D^I(f)$ is homogeneous of degree $d - |I|$.  
\item We have
\[ D^J(D^I (f)) = {I+J \choose I} D^{I+J}(f). \]
\end{enumerate}
\end{prop}

\begin{lemma}[Leibniz Rule]
\label{LEIBNIZLEMMA}
Let $n =1$, $i \in \N$ and $g,h \in R[t]$.  Then we have 
\begin{equation}
\label{LEIBNIZLEMMAEQ}
 D^i(gh) = \sum_{j=0}^i D^j(g) D^{i-j}(h). 
\end{equation}
\end{lemma}
\begin{proof}
Step 1: Recall Vandermonde's Identity: for $m,n,r \in \N$, we have 
\[ {m+n \choose i} = \sum_{j=0}^i {m \choose j}{n \choose i-j}. \]
(If we have a set consisting of $m$ red balls and $n$ blue balls, then for $0 \leq j \leq i$, the number of $i$ element 
subsets containing exactly $j$ red balls is ${m \choose j}{n \choose i-j}$, so $\sum_{j=0}^i {m \choose j}{n \choose i-j}$ is the total number of $i$ element subsets of an $m+n$ element set.) \\
Step 2: 
Using Vandermonde's Identity, we get
\[ D^i(t^m t^n) = {m+n \choose i} t^{m+n-i} =   \sum_{j=0}^i{m \choose j}{n \choose i-j} t^{m+n-i} =  \sum_{j=0}^i D^j(t^m) D^{i-j}(t^{m+n-i}). \]
By $R$-linearity of the Hasse Derivatives, this establishes (\ref{LEIBNIZLEMMAEQ}).
\end{proof}

\subsection{Multiplicities}
\label{sec:multiplicities}
\textbf{} \\ \\ \noindent
Let $f \in R[\underline{t}]$ be nonzero and $x \in R^n$.  The
\textbf{multiplicity of f at x},
denoted $m(f,x)$, is the natural number $m$ such that $D^J(f)(x) = 0$
for all $J$ with $|J| < m$ and $D^J(f)(x) \neq 0$ for some $J$ with $|J| = m$.  We put $m(0,x) = \infty$ for all $x \in R^n$.

\begin{lemma}[{\cite[Lemma 2.4]{DKSS13}}]
\label{0.5}
For $f \in R[\underline{t}]$, $x \in R^n$ and $I \in \N^n$, we have
\[ m(D^I (f),x) \geq m(f,x) - |I|. \]
\end{lemma}
\noindent
Given a vector $\underline{f} = (f_1,\dots,f_k) \in R[\underline{t}]^k$, we put $m(\underline{f},x) = \min_{1 \leq j \leq k}
m(f_j,a)$.

\begin{prop}[{\cite[Prop. 2.5]{DKSS13}}]
\label{DKSS2.5}
\label{0.6}
Let $X_1,\dots,X_n,Y_1,\dots,Y_{\ell}$ be independent indeterminates.
Let $\underline{f} = (f_1,\dots,f_k) \in R[X_1,\dots,X_n]^k$ and
let $\underline{g} = (g_1,\dots,g_n) \in R[Y_1,\dots,Y_{\ell}]^n$.
We define $\underline{f} \circ \underline{g} \in R[Y_1,\dots,Y_{\ell}]^k$
to be $\underline{f}(g_1,\dots,g_n)$.  
\begin{enumerate}[$(a)$]
\item For any $a \in R^{\ell}$ we have
\[ m(\underline{f} \circ \underline{g},a) \geq m(\underline{f},\underline{g}(a)) m(\underline{g}-\underline{g}(a),a). \]
\item We have
\[ m(\underline{f} \circ \underline{g},a) \geq m(\underline{f},\underline{g}(a)). \]
\end{enumerate}
\end{prop}

\begin{cor}[{\cite[Cor. 2.6]{DKSS13}}]
\label{DKSS2.6}
\label{0.7}
Let $f \in R[\underline{t}]$ and let $a,b \in R^n$.  Then for all $c \in R$ we have
\[ m(f(a+tb),c) \geq m(f,a+cb). \]
\end{cor}

\begin{lem}
\label{lem:mult_Leibniz}
Let $f, g, h \in R[t]$ be nonzero polynomials such that $f = gh$ and let $x \in R$.
If $g(x)$ is not a zero divisor, then $m(f, x) = m(h, x)$. 
\end{lem}
 \begin{proof}
Let $m := m(f, x) \geq 1$ (for $m = 0$ the result is easily proved). By Lemma \ref{LEIBNIZLEMMA}, for all $i \in \N$ we have 
 \[D^i(f) = D^0(g)  D^i(h) + \dots + D^i(g) D^0(h).\]
Thus $D^i(f)(x) = 0$ for all $i < m(h,x)$ and $m(f,x) \leq m(h,x)$.   \\ \indent
We now show that $D^i(h)(x) = 0$ for all $i < m$, which will give $m(h, x) \geq m(f, x)$ and complete the proof. 
For $i = 0$, we have $0 = f(x) = g(x) h(x)$, and since $g(x)$ is not a zero divisor, we must have $h(x) = D^0(h)(x) = 0$. 
Now $D^1(f)(x) = D^0(g)(x)D^1(h)(x) + D^1(g)(x) D^0(h)(x) = g(x) D^1(h)(x)$. 
Again, since $g(x)$ is not a zero divisor, $D^1(f)(x) = 0$ if and only if $D^1(h)(x) = 0$. 
Continuing in this way, at the $i$th step we have $D^j(h)(x) = 0$ for all $j < i$ and thus $D^i(f)(x) = g(x) D^i(h)(x)$. 
Since $g(x)$ is not a zero divisor and $i < m$, we see that $D^i(h)(x) = 0$. 
 \end{proof}

\begin{lem}
\label{SOUPEDROOTFACTORLEMMA}
\label{lem:mult_single_ring}
Let $R$ be a ring, and let $f \in R[t]$ be a polynomial of degree $d \geq 1$.  
Let $A= \{x_1,\dots,x_n\} \subseteq R$ be a finite set satisfying Condition (D).  Then
\[ \sum_{x \in A} m(f,x) \leq d. \]
\end{lem}
 \begin{proof}
We will prove this by induction on $n$.  
For $n = 1$, we have that $(t - x_1)^{m(f, x_1)}$ divides $f$ and hence $\deg f \geq m(f, x_1)$. 
So suppose $n \geq 2$. 
Write $f = (t - x_n)^{m(f, x_n)} g$. 
Since $A$ satisfies Condition (D), the element $(x_i - x_n)^{m(f, x_i)}$ of $R$ is not a zero divisor for any $i \in [n-1]$. 
Therefore, by Lemma \ref{lem:mult_Leibniz} we have $m(f, x_i) = m(g, x_i)$ for all $i \in [n-1]$. 
By induction hypothesis we get \[\sum_{i=1}^{n-1} m(f,x_i) = \sum_{i = 1}^{n - 1}m(g, x_i)  \leq \deg g = \deg f - m(f_, x_n), \] from which the result follows. 
 \end{proof}

\begin{remark}
An earlier draft of this work contained a different proof of Lemma \ref{SOUPEDROOTFACTORLEMMA}.  In place of 
Lemma \ref{lem:mult_Leibniz}, we used the following: if $f = \prod_{i=1}^m (t-x_i), \ g = \prod_{j=1}^m (t-y_j) \in R[t]$ 
and $x_i - y_j \in R^{\times}$ for all $i$ and $j$, then $f$ and $g$ generate the unit ideal of $R[t]$.  This was proved using some commutative algebra (localization and Nakayama's Lemma) and was thus a bit of a departure from the rest of 
the paper.  The interested reader can find the details in \url{https://arxiv.org/pdf/1508.06020v1.pdf}.
\end{remark}

\begin{lemma}[DKSS Lemma]
\label{DKSS}
Let $A = \prod_{i=1}^n A_i \subset R^n$ be a finite subset satisfying
Condition (D).  Let $f \in R[\underline{t}]$, and write
\[ f= \sum_{j=0}^{d_n} f_j(t_1,\dots,t_{n-1}) t_n^j \]
with $f_{d_n} \neq 0$.  Put $A' = \prod_{i=1}^{n-1} A_i$.
For all $x' = (x_1,\dots,x_{n-1}) \in A'$, we have
\[ \sum_{x \in A_n} m(f,(x',x)) \leq (\# A_n) m(f_{d_n},x') + d_n. \]
\end{lemma}
\begin{proof}
Choose $I' \in \N^{n-1}$ such that $|I'| = m(f_{d_n},x')$ and
$D^I(f_{d_n})(x') \neq 0$.  Put $I = I' \times \{0\} \in \N^n$.  Then
\[ D^I(f) = \sum_{j=0}^{d_n} D^{I'}(f_j) t_n^j, \]
so $D^I f \neq 0$.
By Lemma \ref{0.5}, we have
\[ m(f,(x',x)) \leq |I| + m(D^I(f), (x',x)) = m(f_{d_n},x') + m(D^I(f),(x',x)). \]
Apply Corollary \ref{0.7} to $D^I(f)$ with $a = (x',0)$, $b = (0,1)$ and
$c = x$: we get
\[ m(D^I(f),(x',x)) \leq  m(D^I (f)(x',t_n),x). \]
Summing over $x \in A_n$ gives
\[ \sum_{x \in A_n} m(f,(x',x)) \leq (\# A_n) m(f_{d_n},x') +
\sum_{x \in A_n} m(D^I(f)(x',t_n),x). \]
Since $I = I' \times \{0\}$, $D^I(f)(x',t_n)$
has degree $d_n$ and thus Lemma \ref{SOUPEDROOTFACTORLEMMA} gives
\[ \sum_{x \in A_n} m(D^I(f)(x',t_n),x) \leq d_n. \]
The result follows.
\end{proof}

\begin{remark}
The case of Lemma \ref{DKSS} in which $R$ is a field and
$A_1 = \dots = A_n$ is due to Dvir, Kopparty, Saraf and
Sudan \cite[pp. 8-9]{DKSS13}.  Our proof follows theirs very closely,
but uses Lemma \ref{SOUPEDROOTFACTORLEMMA} in place
of the root-factor phenomenon.
\end{remark}

\subsection{Multiplicity Enhanced Schwartz Theorem}
\label{sec:mult_Schwartz}
\begin{thm}[Multiplicity Enhanced Schwartz Theorem]
\label{MULTISCHWARTZ}
Let $R$ be a ring, let $A = \prod_{i=1}^n A_i \subset R^n$ be finite, nonempty
and satisfy Condition (D), and let $f = f_n \in F[t_1,\dots,t_n]$ be a
nonzero polynomial.  Let $d_n = \deg_{t_n} f$, and
let $f_{n-1} \in R[t_1,\dots,t_{n-1}]$ be the coefficient of $t_n^{d_n}$
in $f_n$.  Let $d_{n-1} = \deg_{t_{n-1}} f_{n-1}$, and let $f_{n-2} \in R[t_1,\dots,t_{n-2}]$
be the coefficient of $t_{n-2}^{d_{n-2}}$ in $f_{n-2}$.  Continuing in this
manner we define for all $1 \leq i \leq n$ a polynomial
$f_i \in R[t_i,\dots,t_n]$ with $\deg_{t_i} f_i = d_i$.   Then
\[ \sum_{x \in A} m(f,x) \leq \# A\sum_{i=1}^n \frac{d_i}{\# A_i}. \]
\end{thm}
\begin{proof}
By induction on $n$.  The case $n = 1$ is Lemma \ref{SOUPEDROOTFACTORLEMMA}.
Suppose the result holds for polynomials in $n-1$ variables.  Let $A' = \prod_{i=1}^{n-1} A_i$.  Applying Lemma \ref{DKSS} and then the induction hypothesis, we get
\[ \sum_{x \in A} m(f,x) = \sum_{x' \in A'} \sum_{x \in A_n} m(f,(x',x))
\leq \# A_n \sum_{x' \in A'}  m(f_{n-1},x') + \# A' d_n \]
\[ \leq \# A_n \# A' \sum_{i=1}^{n-1} \frac{d_i}{\# A_i} + \# A \frac{d_n}{\# A_n} =
\# A \sum_{i=1}^n \frac{d_i}{\# A_i}. \qedhere\]
\end{proof}

\begin{thm}[Multiplicity Enhanced Generalized Schwartz--Zippel Lemma]
Let $A = \prod_{i=1}^n A_i \subset R^n$ be a finite grid satisfying Condition (D), and suppose $\# A_1 \geq \dots \geq \# A_n$.  Let $f \in R[\underline{t}] = R[t_1,\dots,t_n]$ be a nonzero polynomial.  Then
\[ \sum_{x \in A} m(f,x) \leq \deg f \prod_{i=1}^{n-1} \# A_i. \]
\end{thm}
\begin{proof}
This follows from Theorem \ref{MULTISCHWARTZ} as Theorem
\ref{LOWDEGREEAFTHM} does from Theorem \ref{SCHWARTZ}.
\end{proof}

\begin{remark}
\begin{enumerate}[$(a)$]
\item When $R$ is a field, Theorem \ref{MULTISCHWARTZ} was proved by Geil and Thomsen \cite[Thm. 5]{Geil-Thomsen13}.  They also build closely on \cite{DKSS13}. 
\item Unlike most of the other results presented here, Theorem \ref{MULTISCHWARTZ} is \emph{not} claimed to be sharp in all cases.  In fact, it is not always sharp, and Geil and Thomsen give significant discussion of this point including an algorithm which sometimes leads to an improved bound \cite[Thm. 6]{Geil-Thomsen13} and further numerical exploration.
\end{enumerate}
\end{remark}

\subsection{A Counterexample}
\label{sec:counterexample}
It is natural to ask whether Alon--F\"uredi holds in multiplicity enhanced form, i.e., whether the bound
\[ \# Z_A(f) \leq \# A - \mm(\# A_1,\dots,\# A_n;\sum_{i=1}^n \# A_i - \deg f) \]
could be improved to
\[ \sum_{x \in A} m(f,x) \leq \# A - \mm(\# A_1,\dots,\# A_n;\sum_{i=1}^n \# A_i - \deg f). \]
The following example shows that such an improvement does not always hold.

\begin{example}
\label{EXAMPLE3}
Let $n = 2$ and $R = A_1 = A_2 = \F_q$.
Let $d_1,d_2 \in \Z^+$ be such that $d_1, d_2 < q \leq d_1 + d_2 + 1$. Then $f = t_1^{d_1} t_2^{d_2}$ is $A$-reduced, and we have
\[ \sum_{x \in A} m(f,x) = q d_1 + q d_2 > q^2 - 2q + d_1 + d_2 + 1 =
q^2 - \mathfrak{m}(q,q;2q - d_1 - d_2). \]
\end{example}
\noindent
Notice that the polynomial $f = t_1^{d_1} t_2^{d_2}$ is polylinear (see $\S$\ref{sec:Petrov}).  So far as we know it may be true that $\sum_{x \in A} m(f,x)$ is maximized among all polynomials of fixed degree when $f$ is a polylinear polynomial.  We leave this as an open problem.

\end{document}